\documentclass{amsart}
\usepackage[english]{babel}
\usepackage{amsthm,amsmath,amsfonts,amssymb,enumerate,array,multirow,booktabs,stmaryrd,esint}
\usepackage{hyperref}
\newtheorem{thm}{Theorem}[section]
\newtheorem{lemma}[thm]{Lemma}
\newtheorem{corol}[thm]{Corollary}
\newtheorem{prop}[thm]{Proposition}

\theoremstyle{definition}
\newtheorem{defi}[thm]{Definition}
\theoremstyle{remark}

\newtheorem{oss}[thm]{Remark}

\newcommand{\C}{\mathbb C}
\newcommand{\R}{\mathbb R}

\newcommand{\T}{\mathbb T}

\newcommand{\dez}[1]{\frac{\partial}{\partial #1}}

\newcommand{\lie}{\mathfrak}

\DeclareMathOperator{\Iso}{Iso}

\DeclareMathOperator{\Aut}{Aut}
\DeclareMathOperator{\End}{End}
\DeclareMathOperator{\rea}{Re}
\DeclareMathOperator{\img}{Im}
\DeclareMathOperator{\id}{id}
\DeclareMathOperator{\Ric}{Ric}
\DeclareMathOperator{\Diffeo}{Diffeo}

\DeclareMathOperator{\const}{const}
\DeclareMathOperator{\vol}{vol}
\renewcommand{\phi}{\varphi}
\renewcommand{\epsilon}{\varepsilon}

\renewcommand{\bar}{\overline}
\renewcommand{\tilde}{\widetilde}
\renewcommand{\Re}{\rea}
\renewcommand{\Im}{\img}
\newcommand{\lieder}{\mathcal L}

\newcommand{\desharp}{\partial^\sharp}
\newcommand{\de}{\partial}
\newcommand{\debar}{{\bar \partial}}
\DeclareMathOperator{\grad}{grad}

\newcommand{\restr}{ |_{(0,0,0)} }
\newcommand{\param}{{\alpha, \phi}}
\newcommand{\call}{\mathcal}
\newcommand{\nablat}{\nabla^T}

\newcommand{\Fol}{\mathrm{Fol}}
\newcommand{\fol}{\lie{fol}}
\newcommand{\CinfB}{C^\infty_{\textup B}}
\newcommand{\deB}{\partial_{\textup B}}
\newcommand{\dB}{d_{\textup B}}
\newcommand{\deltaB}{\delta_{\textup B}}
\newcommand{\debarB}{\bar \partial _{\textup B}}
\newcommand{\laplB}{\Delta_{\textup B}}
\newcommand{\basicforms}{\Omega_{\textup B}}
\newcommand{\HB}{H_{\textup B}}
\newcommand{\cB}{c_1^\textup{B}}
\newcommand{\paramt}{{t, \alpha, \phi}}
\renewcommand{\param}{\paramt}
\title{On Sasaki-Ricci solitons and their deformations}
\author{David Petrecca}
\address{Dipartimento di Matematica\\ Universit\`a di Pisa \endgraf Largo Pontecorvo 5 \endgraf 56127 Pisa, Italy}
\email{petrecca at mail.dm.unipi.it}

\begin{document}
\maketitle
\begin{abstract}We extend to the Sasakian setting a result of Tian and Zhu about the decomposition of the Lie algebra of holomorphic vector fields on a K\"ahler manifold in the presence of a K\"ahler-Ricci soliton. Furthermore we apply known deformations of Sasakian structures to a Sasaki-Ricci soliton to obtain a stability result concerning \emph{generalized} Sasaki-Ricci solitons, generalizing Li in the K\"ahler setting and also He and Song by relaxing some of their assumptions.
\end{abstract}

\section*{Introduction}
Sasakian geometry is often referred to as the odd dimensional analogue of K\"ahler geometry. A Sasakian structure sits between two K\"ahler structures, namely the one on its Riemannian cone and the one on the normal bundle of its Reeb foliation. It is then natural to ask whether known results in K\"ahler geometry extend to the Sasakian setting.

A prominent role is played by Sasaki-Einstein manifolds, also due to their application in physics in the so-called AdS/CFT correspondence. There is a large number of examples and techniques to build Sasaki-Einstein manifolds, see e.g. \cite[Chap.~5]{monoBG} and \cite{sparks_survey} and the references therein. As an example of interrelation between the Sasakian structure and the two K\"ahler structures, we mention that a manifold is Sasaki-Einstein if and only if its transverse K\"ahler structure is K\"ahler-Einstein if and only if the Riemannian cone is Ricci-flat.

A possible generalization of Sasaki-Einstein metrics is to consider \emph{transverse K\"ahler-Ricci solitons}, also known as \emph{Sasaki-Ricci solitons}. A Sasakian structure $(\eta, g)$ on $M$ is said to be a Sasaki-Ricci soliton if there exists a Hamiltonian holomorphic complex vector field $X$ on $M$ (see Definition \ref{def:hamholo}) such that the Ricci form $\rho^T$ and the transverse K\"ahler form $\omega^T = \frac 1 2 d\eta$ satisfy
\begin{equation}\label{SRS}
\rho^T - (2n+2) \omega^T = \lieder_X \omega^T.
\end{equation}

Their K\"ahlerian counterparts are then metrics whose Ricci and K\"ahler forms satisfy $\rho - \omega = \lieder_X \omega$ for a vector field $X$ on $M$ which turns out to be \emph{holomorphic}. They have been extensively studied, one possible motivation is that they are special solutions of the K\"ahler-Ricci flow (see e.g. \cite{chow} and the references therein, see also \cite{SRF} for an introduction to the transverse K\"ahler-Ricci flow, called \emph{Sasaki-Ricci flow}). 

The presence of a K\"ahler-Ricci soliton on a compact K\"ahler manifold $M$ can give information about the Lie groups and Lie algebras of transformations of $M$. Namely Tian and Zhu prove, among other things, the following theorem.
\begin{thm}[\cite{tz_uniq}] The Lie algebra $\lie h$ of holomorphic vector fields on a K\"ahler-Ricci soliton with vector field $X$ admits the splitting
\[
\lie h = \lie h_0 \oplus \bigoplus_{\lambda>0} \lie h_\lambda
\]
where $\lie h_0$ is the centralizer in $\lie h$ of $X$ and in turn splits $\lie h_0 = \lie h_0' \oplus \lie h_0''$ as the space of $(1,0)$-gradients of real and purely imaginary functions and $\lie h_\lambda = \{ Y \in \lie h: [X, Y] = \lambda Y \}$.

The space $\lie h_0''$ is identified with the Lie algebra of the isometry group $\Iso(M)$, whose connected component is maximal compact in the connected component $\Aut(M)^0$ of the group of holomorphic automorphisms.
 \end{thm}
One of the results in this paper, namely Theorem \ref{thm:decomp}, is a partial extension to the Sasakian setting of the above theorem. It concerns the decomposition of a \emph{quotient} Lie algebra of transverse infinitesimal transformations of the Sasaki manifold. This is done in order to keep some consistency with the Sasaki-extremal case (see \cite{canonical, vancov}) and with some older decomposition results by Nishikawa and Tondeur \cite{NishTond} holding in the more general setting of transversely K\"ahler foliations with minimal leaves. 

Then we shall consider an application to Sasaki-Ricci solitons of the deformation theory of Sasakian structures. We will start with a given Sasaki-Ricci soliton and a connected group $G$ of Sasaki automorphism and deform the structure in a $G$-equivariant way through three standard types of deformations, namely type I, type II and the ones of the transverse complex structure introduced by Nozawa in \cite{nozawa}.

This has already been done in the Sasaki-extremal case by van Coevering in \cite{vancov} who proved, under some assumptions on the Futaki invariant, a stability result for extremal Sasakian metrics with vanishing reduced scalar curvature, reduced in a certain sense. His method of proof makes use of the implicit function theorem, whose assumptions are satisfied thanks to the assumption on the Futaki invariant.

Our result is the solitonic analogue of van Coevering's work. We will prove in Theorem \ref{thm:defoSRS} the following fact. Given a Sasaki-Ricci soliton together with a connected group $G$ of Sasaki automorphisms, we prove the existence, after $G$-equivariant deformations as above, of a wider class of Sasakian metrics known as  \emph{generalized} Sasaki-Ricci solitons (cf. Definition \ref{def:genSRS}).
This result is also a generalization of a result in the K\"ahler setting, obtained by Li in \cite{Li}. 

In the paper \cite{frankel} by He and Sun it is proved, among other things, a stability result for Sasaki-Ricci solitons under $\T$-equivariant type I and II deformations, for a fixed torus $\T$ of Sasaki transformations. Our work can be also thought as a generalization of that as we have relaxed the possible choices for the group, added more deformations and obtained a wider class of metrics as output.

The paper is organized as follows. We start by recalling some background in Sasakian geometry in Section \ref{sec:background}, then we define and state some properties of Hamiltonian holomorphic fields in Section \ref{sec:norm}. In Section \ref{sec:inftransf} we prove the decomposition theorem and finally in Section \ref{sec:deform} we state and prove the theorem about deformations.

\subsection*{Acknowledgements}The author would like to thank Song Sun and Craig~van Coevering for the discussions and clarifications and Fabio~Podest\`a for suggesting the problem, for his constant advice and support and also for his help in a better presentation of this paper.

\section{Background} \label{sec:background}
Here we recall the definitions and main facts about Sasakian structures and their deformations which will be used throughout the paper. We refer to the monograph \cite{monoBG} or more detailed references will be given as needed.
\subsection{Sasakian manifolds}
\begin{defi}
A Riemannian manifold $(M, g)$ is \emph{Sasakian} if its metric cone $C(M)=M \times \R^+$ with  metric $\bar g = r^2 g + dr^2$ is K\"ahler, with $r$ the coordinate on $\R^+=(0, +\infty)$. 
\end{defi}
$M$ has odd dimension $2n+1$ and is identified with $M \times \{ 1\}$.
We have an integrable complex structure $I \in \End(T C(M))$. We use the Euler field $r \partial_r$ on the cone to define the vector field $\xi = I(r \partial_r)$ which is tangent to $M$. Both $\xi$ and $r \partial_r$ are real holomorphic with respect to $I$ and $\xi$ is Killing with respect to both $g$ and $\bar g$. Moreover $\xi$ is unitary and with geodesic orbits on $M$.

Consider also the $1$-form $\eta = d^c \log r$ on $C(M)$. We call their restrictions to $M$ using the same notations, $\eta$ and $\xi$.

The symplectic form $\omega$ on the cone has the property that $\omega = \frac 1 2 d(r^2 \eta)$, hence $\eta$ on $M$ is a contact form and $\xi$ is its Reeb field, i.e. $\eta(\xi)=1$ and $\iota_\xi d\eta = 0$. Moreover it can be easily seen that they are the Riemannian dual of each other.

The contact form points out a non-integrable distribution $D = \ker \eta$. It gives a $g$-orthogonal splitting 
\[
TM = D \oplus L_\xi
\]
where $L_\xi$ is the trivial line bundle generated by $\xi$.
Restrict $I$ to $D$ and denote it by $J$. We then extend it to an endomorphism $\Phi \in \End(TM)$ by setting $\Phi \xi = 0$. From the fact that $I$ is a complex structure and that $\bar g$ is Hermitian it follows that
\begin{align}
\Phi^2 				& = -\id + \eta \otimes \xi  \label{eq:phi2} \\
g(\Phi \cdot, \Phi \cdot) 	&= g - \eta \otimes \eta. \label{eq:compmetric}
\end{align}

Equation \eqref{eq:phi2} says that the triple $(\eta, \xi, \Phi)$ is an \emph{almost contact structure} on $M$ and \eqref{eq:compmetric} says that the metric $g$ is a \emph{compatible metric} making $M$ a \emph{contact metric structure}. The fact that $\xi$ is Killing says that the structure is \emph{K-contact}. Moreover \eqref{eq:compmetric} also says that the restriction $g|_D$ is $J$-Hermitian and another property is that
\begin{equation}\label{deta}
\frac 1 2 d\eta = g(\Phi \cdot, \cdot),
\end{equation}
that is $\frac 1 2 d\eta|_D$ is the fundamental form of $g|_D$. Moreover $(D, J)$ is a pseudoconvex CR structure. 

We shall denote a Sasakian structure on $M$ by a tuple $\call S = (\eta, \xi, \Phi, g)$ of tensor fields as above.

The orbits of the vector field $\xi$ define on $M$ a foliation $\call F = \call F_\xi$ called the \emph{Reeb} or \emph{characteristic} foliation. If its leaves are compact  the Sasakian manifold is called \emph{quasi-regular}, otherwise \emph{irregular}. In the former case, if the induced circle action is free,  the manifold is called \emph{regular}.

The subbundle $D$ is identified via $g$ with the normal bundle $\nu \call F := TM/T\call F$ of the Reeb foliation. By the consideration above we see that $\nu(\call F)$ is endowed with an integrable complex structure $J = \Phi|_D$ and a symplectic form $\frac 1 2 d\eta$. We then define  the metric on $D$
\[
g^T= \frac 1 2 d\eta(\cdot, J \cdot)
\]
and we obtain a  \emph{transverse K\"ahler structure} on $M$. The transverse metric $g^T$ and $g$ are related by $g = g^T + \eta \otimes \eta$.

It can be proven that the connection
\[
\nablat_X Y = \begin{cases}	(\nabla_X Y)^D 	& \text{ if } X \in D \\
						[X,Y]^D			& \text{ if } X \in L
			\end{cases}
\]
is the unique torsion-free connection compatible with the metric $g^T$ on $D$. Out of this we define the transverse Riemann curvature tensor 
\[
R^T_{X,Y} Z = \nablat_X \nablat_Y Z - \nablat_Y \nabla_X Z - \nablat_{[X,Y]} Z.
\]
Let $\Ric^T$ and $s^T$ be the transverse Ricci tensor and the transverse scalar curvature defined by averaging from $R^T$.
With computations similar to the ones for Riemannian submersions (e.g. \cite{besse}) one can prove the following. 
\begin{prop}[\cite{monoBG}]
For a Sasakian metric $g$ on a manifold of dimension $2n+1$ we have the following.
\begin{enumerate}[(i)]
\item $\Ric_g(X, \xi) = 2n \eta(X)$ for $X \in \Gamma(TM)$;
\item $\Ric^T(X,Y) = \Ric_g(X,Y) + 2 g^T(X,Y)$ for $X, Y \in \Gamma(D)$;
\item $s^T = s_g + 2n$.
\end{enumerate}
\end{prop}

A differential $r$-form $\alpha$ is said to be \emph{basic} if
\[
\iota_\xi \alpha = 0, \qquad \iota_\xi d\alpha = 0.
\]
The space of global basic $r$-forms will be denoted by $\basicforms^r(M)$. In particular a function $f$ is basic if $\xi \cdot f = 0$. We denote by $\CinfB(M, \R)$ the space $\basicforms^0(M)$ of real valued basic functions.
The exterior derivative maps basic forms to basic forms so it makes sense to consider the subcomplex $(\basicforms^*(M), \dB)$ of the usual deRham complex, where $\dB$ is the restriction of the exterior derivative to basic forms. Its cohomology is called \emph{basic deRham cohomology} and is denoted by $\HB^*(M)$.

One can consider also the adjoint $\deltaB$ of $\dB$ by means of a transverse Hodge star operator and the corresponding \emph{basic Laplacian} $\laplB = \deltaB \dB + \dB \deltaB$. 

From the fact that the transverse geometry is K\"ahler  the \emph{basic $(p,q)$-forms} can be defined. As it is done in classical complex geometry, the transverse complex structure induces a splitting of the bundle of $r$-forms and  consequently a splitting 
\[
\Omega^r_{\textup B} = \bigoplus_{p+q=r} \Omega^{p,q}_{\textup B}.
\]
One can also construct the basic Dolbeault operators $\deB$ and $\debarB$ that will share the properties of the usual ones on complex manifolds, namely $\dB = \deB + \debarB$ and also $\deB^2 = \debarB^2 = \deB \debarB + \debarB \deB = 0$. Also, let $\dB^c = i(\debarB-\deB)$. The cohomology of the complex $(\basicforms^{*,*}(M), \debarB)$ is called \emph{basic Dolbeault cohomology} and denoted by $\HB^{*,*}(M)$.
See \cite[Chap.~7]{monoBG} for details.

A particular basic cohomology class in $\HB^{1,1}(M)$ is the \emph{first basic Chern class}, defined by the basic class $\cB(M):= [\frac 1 {2 \pi} \rho^T]_B$ where $\rho^T = \Ric^T(\Phi \cdot, \cdot)$ is the transverse Ricci form.

A Sasakian manifold is called \emph{transversally Fano} or \emph{positive} if this class is represented by a positive basic $(1,1)$-form.
We recall a well known result useful for later.
\begin{prop} [\cite{monoBG}] \label{c1D} The real first Chern class $c_1(D)$ of the vector bundle $D$ vanishes (is a torsion class) if, and only if, there exist $a \in \R$ such that $\cB(M) = [ad\eta]_B$.
\end{prop}
\subsection{Transversely holomorphic fields and holomorphy potentials}
We have the group of the automorphisms of the foliation given by
\[
\Fol(M, \xi) = \{ \phi \in \Diffeo(M): \phi_* \call F_\xi \subseteq \call F_\xi \}
\]
and its Lie algebra 
\[
\fol(M, \xi) = \{ X \in \Gamma(TM): [\xi,X] \in \Gamma(L_\xi) \}
\]
also called the space of \emph{foliate vector fields}.

For any foliate vector field $X$ on $M$ we denote by $\bar X$ its projection onto the space of sections of the normal bundle $\nu (\call F_\xi)$.
The image of such projections has a Lie algebra structure defined by $[ \bar X, \bar Y] := \bar{[X,Y]}$. We call it the algebra of \emph{transverse vector fields}.

Since the differential of an automorphism of the foliation is an endomorphism of the normal bundle $\nu (\call F_\xi)$ which has a complex structure $J$, we can also define the group of \emph{transversally holomorphic} transformations by
\[
\Fol(M, \xi, J ) = \{ \phi \in \mathrm{Fol}(M, \xi): \phi_* J = J \phi_*\}
\]
with Lie algebra
 called the space of \emph{transversally holomorphic} vector fields.
This space can be expressed as 
\[
\fol(M, \xi, J) = \{ X \in \fol(M,\xi): \bar{[X, \Phi Y]} = J \bar{[X,Y]} \text{ for all }Y \in \Gamma(TM) \}
\]
as $\Phi Y$ is a representative of $J \bar Y$.

As in the classical case, the transverse complex structure gives a splitting $\nu(\call F_\xi) = \nu(\call F_\xi)^{1,0} \oplus \nu(\call F_\xi)^{0,1}$.
 
Given a complex valued basic function $u \in \CinfB(M, \C)$, we define $\desharp_g u$ to be the $(1,0)$-component of the gradient of $u$, i.e. the transverse field such that
\[
g(\desharp_g u, \cdot) = \debar u
\]
or simply $\desharp u$ if the metric is clear from the context.

The field $\desharp u$ need not to be transversally holomorphic. The space of basic function that give rise to transversally holomorphic fields is the kernel $\call H_g$ of the fourth order elliptic operator $L_g=(\debar \desharp_g)^*(\debar \desharp_g)$ which, as in the K\"ahler case, can be expressed as
\[
L_g u = \frac 1 4 \biggl (\laplB^2 u + (\rho^T, dd^c u) + 2 (\debarB u, \debarB s^T) \biggr ).
\]

\subsection{Deformations of Sasakian structures}

Let $\call S=(\eta, \xi, \Phi, g)$ be a Sasakian structure on $M$. We now want to keep the Reeb field $\xi$ fixed and let $\eta$ vary by perturbing it with a basic function. Namely, for $\phi \in \CinfB(M, \R)$ we let
\[
\tilde \eta = \eta + \dB^c \phi.
\]
Then we have $d\tilde \eta = d\eta + \dB \dB^c \phi$ and, for small $\phi$, the form $\tilde \eta$ is still contact.
We have the following.
\begin{prop}For a small basic function $\phi$ there exists a Sasakian structure on $M$ with same Reeb field $\xi$, same holomorphic structure on $C(M)$, same transverse holomorphic structure on $\call F_\xi$ and contact form $\tilde \eta = \eta + \dB^c \phi$.
\end{prop}

So the deformation $\eta \mapsto \tilde \eta$ deforms the transverse K\"ahler form $d\eta$ in the same transverse K\"ahler class $[d\eta]_B$. The other tensors vary as follows
\begin{align*}
\tilde \Phi 	&= \Phi - \xi \otimes d^c \phi \circ \Phi \\
\tilde g	&= d\tilde \eta \circ (\tilde \Phi \circ \id) + \tilde \eta \otimes \tilde \eta.
\end{align*}
Transverse K\"ahler deformations are a special case of the so called \emph{type II} deformations, in the terminology of \cite{monoBG}.

A second type of deformation  keeps the CR structure fixed but deforms the Reeb foliation. So we define a tuple $\call S' = (\eta', \xi', \Phi', g')$ by
\begin{align} \label{defoI}
\tilde \eta	= f \eta	&,  &\tilde \xi	= \xi + \rho
\end{align}
that is by adding to $\xi$ an infinitesimal automorphism of the Sasakian structure $\rho \in \lie{aut}(\xi, \eta, \Phi, g) = \{ X \in \Gamma(TM): \lieder_X g = 0, \lieder_X \eta = 0, \lieder_X \Phi = 0, [X, \xi]=0 \}$.
The Reeb condition forces $f$ to be equal to $(1 + \eta(\rho))^{-1}$.

As long as $\eta(\tilde \xi) >0$ we still have a contact structure whose contact subbundle $D$ is unchanged as $\Phi|_D = \tilde \Phi|_D$. We extend $\tilde \Phi$ by $\tilde \Phi = \Phi - \Phi \tilde \xi \otimes \tilde \eta$. This will satisfy the compatibility condition in the definition of an almost contact structure. Finally define the Riemannian metric
\[
\tilde g (X,Y)= d \tilde \eta (\tilde \Phi X, Y) + \tilde \eta(X) \tilde \eta(Y)
\]
which is compatible by construction. 
Let us denote by $\lie F(D, J)$ the space of all K-contact structures having $(D, J)$ as underlying almost CR structure.
\begin{defi} A deformation as in \eqref{defoI} within $\lie F(D, J)$ is said to be \emph{of type I}.
\end{defi}

Let us now relate such space with the Lie algebra of infinitesimal CR automorphisms $\lie{cr}(D, J) = \{ X \in \lie{aut}(D) : \lieder_X J = 0 \}$.
Fix a pseudo convex CR structure $(D, J)$ and assume it is what is called \emph{of Sasaki type} i.e. there exists a K-contact structure that admits it as underlying CR structure. This means that the set $\lie F(D, J)$ is non empty and we fix a structure $\call S_0 = (\eta_0, \xi_0, \Phi_0, g_0)$ in it.

\begin{prop} A contact metric structure $\call S = (\eta, \xi, \Phi, g)$ lies in $\lie F(D, J)$ if and only if $\xi \in \lie{cr}(D, J)$.
\end{prop}
We now identify the set $\lie F(D, J)$ with a cone of $\lie{cr}(D, J)$. Namely, define
\[
\lie{cr}^+(D, J) = \{ \xi \in \lie{cr}(D,J) : \eta_0(\xi)>0 \}.
\]
It is a convex cone in the Lie algebra $\lie{cr}(D, J)$ and moreover it is invariant by the adjoint action of the group of CR transformations. This helps to give the following description of  $\lie F(D, J)$.
\begin{prop}
The map $(\xi, \eta, \Phi, g) \mapsto \xi$ defines a bijection $\lie F(D, J) \simeq \lie{cr}^+(D, J)$.
\end{prop}

In particular we are interested in type I deformation giving rise to Sasakian structures that are invariant under the action of a fixed Lie group. So we will have $\alpha$ vary in the center $\lie z$ of the Lie algebra $\lie g$ of the group $G \subseteq \Aut(\eta, \xi, \Phi, g)$. The perturbed Reeb field $\xi + \alpha$ will belong to the cone
\[
\lie z^+ = \{ \zeta \in \lie z : \eta(\zeta) >0 \}
\]
called the \emph{Sasaki cone of $\lie z$}, in the terminology of \cite{vancov}.

Finally we define here a third type of deformations that change the transverse complex structure of the Reeb foliation keeping it fixed as a smooth foliation. For this we refer to \cite{nozawa, vancov}.

There exist a versal space that parameterizes such deformations, whose tangent space is $H^1_{\debarB} (\call A^{0,\bullet})$ that is the first cohomology of the complex $\debarB: \call A^{0, k} \rightarrow \call A^{0, k+1}$ where $\call A^{0,k}$ is the space of smooth basic forms of type $(0,k)$ with values in $\nu(\call F_\xi)^{1,0}$.

If one wants to consider $G$-invariant deformations, for some group $G$ acting on $M$, one can take the complex $(\debarB, \call A^{0,k}_G)$ of $G$-invariant $\nu(\call F_\xi)^{1,0}$-valued forms on $M$ and its first cohomology $H^1_{\debarB} (\call A^{0, \bullet})^G$ is the tangent space of the versal space of $G$-invariant deformations.
In the following let $\call B$ be a smooth subspace of the versal space of $G$-invariant deformations.

There is a known obstruction due to Nozawa \cite{nozawa} for the existence of Sasakian structures compatible with a given deformation of a transverse K\"ahler foliation.
\begin{defi}[\cite{nozawa}]
A deformation $(\call F_\xi, J_t)_{t \in \call B}$ is said to be of $(1,1)$-type if for all $t \in \call B$ the $(0,2)$-component of $d\eta$ vanishes.
\end{defi}

A result of Nozawa states that this is the only obstruction for the existence of compatible Sasakian metrics.

\begin{thm}[\cite{nozawa}]
Let $(\call F_\xi, J_t)_{t \in \call B}$ be a deformation of the Sasakian structure $( \eta, \xi, \Phi, g)$. Then there exist a neighborhood $V$ of $0$ in $\call B$ such that for $t \in V$ there exist a smooth family of compatible Sasakian structures $(\eta_t, \xi, \Phi_t, g_t)$ such that  $(\eta_0, \xi, \Phi_0, g_0) =  (\eta, \xi, \Phi, g)$ if and only if the deformation restricted to $V$ is of $(1,1)$-type.
\end{thm}

The following corollary will be useful to us.
\begin{corol} \label{cor:11type}
A deformation of a positive Sasakian structure is of $(1,1)$-type. In particular when the original Sasakian metric is Einstein or more generally when its transverse K\"ahler form belongs to $2 \pi \cB$.
\end{corol}

We will consider $G$-equivariant deformations of type I and II applied to a deformation $(\eta_t, \xi, \Phi_t, g_t)$ of $(1,1)$-type as above. Thus we consider the following Sasakian structure

\begin{align}
\eta_\paramt 	&=  (1+\eta_t(\alpha))^{-1} \eta_t + d^c \phi \nonumber \\
\xi_\paramt	&= \xi + \alpha  \label{alldefo} \\ 
\Phi_\paramt	&= \Phi_{t, \alpha} - (\xi + \alpha) \otimes d^c \phi \circ \Phi_{t, \alpha} \nonumber
\end{align}
and $g_\paramt$ given by the compatibility relations, for $(t, \alpha, \phi) \in \call B \times \lie z \times H^k(M)^G$, where $H^k(M)^G$ is the space of $G$-invariant $L^2$ functions whose derivatives up to order $k$ are $L^2$. We assume $k>n+5$ in order to apply the Sobolev imbedding theorem and have the curvatures of $g_\paramt$ well defined.

\section{Normalized Hamiltonian holomorphic vector fields} \label{sec:norm}
We make the same assumption as in \cite{FOW}. Namely we start with a positive compact Sasakian manifold $(M^{2n+1}, g, \eta, \xi, \Phi)$. If we assume $c_1(D)=0$ and normalize, we have $2 \pi \cB=(2n+2)[\frac 1 2 d\eta]_B$ by Proposition \ref{c1D}.

Let $h$ be a Ricci potential, that is a real basic function such that $\rho^T - (2n+2) \frac 1 2 d\eta = i \deB \debarB h$ and consider the operator $\Delta^h$ acting on basic functions as
\[
\Delta^h u = \Delta_{\debarB}u - (\debar u, \debar h).
\]
Here we have dropped the $B$ subscript and we will do the same in the following as it will be clear from the context.
\begin{oss} \label{rmk-lapl}
This is the $\debar$-Laplacian on functions, with respect to the weighted product $\langle f,g\rangle_h = \int_M f \bar g e^h \mu$.
Indeed
\begin{align*}
\langle \debar^* \debar u, v \rangle_h 	&= \langle \debar u, \debar v \rangle_h \\
							&= \int_M \de_{\bar a} u \bar{\de_{\bar b} v} g^{\bar a  b} e^h \mu \\
							&= -  \int_M \biggl ( \nabla_b \nabla^b u \cdot \bar v + \nabla^b u \nabla_b h \cdot \bar v \biggr) e^h \mu \\
							&= \int_M \Delta^h u \cdot \bar v e^h \mu
\end{align*}
with volume form $\mu = (\frac 1 2 d\eta)^n \wedge \eta$.
\end{oss}
We now consider a class of vector fields introduced in \cite{CFO-transvalg, FOW}.
\begin{defi} \label{def:hamholo}
A complex vector field $X$ on $M$ commuting with $\xi$ is called \emph{Hamiltonian holomorphic} if its projection onto the normal bundle $\bar X$ is transversally holomorphic and the basic function, sometimes called \emph{potential}, $u = i \eta(X)$ is such that 
\[
\iota_X \omega^T = i \debarB u.
\]
It is \emph{normalized} if $\int_M u e^h \mu = 0$.
\end{defi}

It must then have the form 
\[
X = -iu \xi + \desharp u = -iu \xi + \nabla^j u e_j,
\]
where $e_j = \dez{z^i} - \eta_i \xi$ generate $D^{1,0} \simeq \nu(\call F_\xi)^{1,0}$.

We recall a widely known fact.
\begin{lemma} \label{lemma:halg} The subset $\lie h = \{X: \text{is Hamiltonian holomorphic} \}$ is a Lie subalgebra of the algebra of the algebra of vector fields on $M$.
\end{lemma}
\begin{proof}
Let $X, Y$ be Hamiltonian holomorphic with potentials $u, v$.

Their bracket is
\begin{equation} \label{eq:bracket}
[X, Y] = -i (Xv - Yu) \xi + [\desharp u, \desharp v]
\end{equation}
using the fact that $u,v$ are basic, the $e_j$'s commute among each other and with $\xi$ and that $d\eta$ has basic type $(1,1)$ so it must vanish when evaluated on two $(1,0)$ fields.
 If we let $w := Xv - Yu$ it is (dropping the B's and the T's for simplicity)
\begin{align*}
\iota_{[X,Y]} \omega &= \lieder_X \iota_Y \omega - \iota_Y \lieder_X \omega \\
				&= \lieder_X (i \debar v) - \iota_Y (i \de \debar u) \\
				&= \iota_X (i \de \debar u) - \iota_Y (i \de \debar u)\\
				&= i \debar (Xv - Yu) \\
				&= i \debar w.
\end{align*}
So $[X,Y]$ is Hamiltonian holomorphic with potential $w$.
\end{proof}

Let  $\Lambda_1$ be the  first eigenspace of $\Delta^h$ with eigenvalue $\lambda_1$.
\begin{thm}[\cite{FOW}] \label{teocorresp}
We have
\begin{enumerate}
\item  \label{firsteig} $\lambda_1 \geq 2m+2$.
\item \label{equal} Equality holds if and only if there exist a nonzero Hamiltonian normalized holomorphic vector field.
\item \label{corresp} The correspondence $\Lambda_1 \rightarrow \lie h_N$ given by  
\[
u \mapsto -iu \xi + \desharp u
\]
where $\lie h_N$ denotes the space of normalized Hamiltonian holomorphic fields, is an isomorphism.
\end{enumerate}
\end{thm}
\begin{proof}
\begin{enumerate}
\item We can replicate the computation made in Futaki's book \cite{FutakiBook} to conclude that
\[
(\lambda - (2m+2)) \| \debar u \|^2_h = \| D_gu \|^2_h \geq 0
\]
for every $u$ in the eigenspace of eigenvalue $\lambda$ and the norms are taken using the weighted $L^2$ product and $D_g: \CinfB(M)^\C \rightarrow \Gamma(\nu(\call F_\xi)^{1,0} \otimes \basicforms^{0,1}(M))$ is the operator such that $\ker D_g = \call H_g$.

\item It means that the map in item \ref{corresp} is surjective.
Let $X$ be a Hamiltonian holomorphic vector field with potential function $u$.
Then the $g^T$-dual of the $(1,0)$-part of $X$ is a $\debar$-closed form $\alpha$ such that $\alpha = \debar u$, which is the same as $\iota_X \omega^T = i \debar u$.
This acts as a function $u$ in the Hodge decomposition $\alpha = \alpha_H + \debar u$ would in the K\"ahler setting.
The same computation in Futaki's book shows that
\[
\nabla_{\bar \jmath}(-\Delta^h u + (2m+2)u) = 0
\]
which means that the function $\Delta^h u - (2m+2)u$ equals some constant $c$.
Integrating this equality it we get
\[
\int_M \Delta^h u \cdot e^h \mu - (2m+2) \int_M u e^h \mu = c \vol_h(M)
\]
which implies $c=0$ if we start with a normalized vector field.
\item Being eigenspaces finite dimensional, we have also injectivity, hence isomorphism. \qedhere
\end{enumerate}
\end{proof}
We can use this correspondence to prove the following.
\begin{prop}The subspace $\lie h_N$ is a Lie subalgebra of $\lie h$.
\end{prop}
\begin{proof}Let $X,Y \in \lie h_N$ be the image of functions $u,v$ via the correspondence. Then from the proof of Lemma \ref{lemma:halg} we have that the potential of $[X,Y]$ is $w = Xv - Yu$. Its integral is
\begin{align*}
\int_M w e^h \mu 	&= \int_M  (\debar v, \debar \bar u) e^h \mu - \int_M (\debar u, \debar \bar v) e^h \mu \\
				&= \int_M \Delta^h v \cdot u e^h \mu - \int_M \Delta^h u \cdot v e^h \mu \\
				&= 0
\end{align*}
where in the last equality we use that $u$ and $v$ are eigenfunctions of $\Delta^h$ with the same eigenvalue $2n+2$ and in the penultimate the self-adjointness of $\Delta^h$ (see e.g. \cite[Eq.~(33)]{FOW}).
\end{proof}

\section{A Lie algebra of infinitesimal transformations and its decomposition} \label{sec:inftransf}
Let there exist a Sasaki-Ricci soliton (SRS for short) as in \cite{FOW}, i.e. a Sasakian metric such that
\begin{equation}\label{eq:SRS_forms}
\rho^T - (2n+2) \tfrac 1 2 d\eta = \lieder_X \tfrac 1 2 d\eta
\end{equation}
 with Hamiltonian holomorphic normalized vector field $X$ and potential $\theta_X$ which, by an easy computation (e.g. \cite{FOW}), is equal to the Ricci potential $h$ up to a constant. The field $X$ can be written as
\begin{equation}
X = -i \theta_X \xi + \desharp h
\end{equation}
with $\int_M \theta_X e^h \mu = 0$.
Let the section of $D^{1,0}$ given by $\desharp h = \desharp \theta_X$ decompose as $\tilde X - i J \tilde X$, where $J$ is the transverse complex structure.

Consider the following operators $L$ and $\bar L$ acting on basic functions.
\begin{align*}
Lu		&= \Delta u - (\debar u, \debar h) - (2n+2) u = \Delta u - \bar X \cdot u - (2n+2)u\\
\bar L u	&= \Delta u - (\debar h, \debar \bar u) - (2n+2)u = \Delta u - X \cdot u - (2n+2)u.
\end{align*}
\begin{lemma}
The operators $L$ and $\bar L$ have the following properties.
\begin{enumerate}
\item $\bar{ L \bar u} = \bar L u$;
\item Each of them is self-adjoint with respect to the weighted $L^2$-product on the space of basic functions.
\item  $L$ and $ \bar L$ commute, so $\bar L$ maps $\ker L$ into itself.
\end{enumerate}
\end{lemma}
\begin{proof}
The first item is just a computation using that the pairing $(,)$ is Hermitian and that the Laplacian is a real operator.

For the second, notice that $L + (2n+2) \id = \Delta^h$ is self-adjoint because it is the $\debar$-Laplacian of the weighted metric as shown in Remark \ref{rmk-lapl}.

For the commutativity, it is enough to show $[L-\bar L, \bar L]=0$. We have
\[
(L-\bar L)u = (X-\bar X) u = 2i \Im(X) u = -2i J\tilde Xu.
\]
This operator commutes with $\bar L$ if and only if $J \tilde X$ commutes with the Laplacian (it is a general fact that Killing fields commute with Laplacians) and that $[J \tilde X, X ]=0$ being $\tilde X$ transversally real holomorphic.
\end{proof}

Let $E_\lambda$ be the eigenspace of $\bar L|_{\ker L}$ of eigenvalue $-\lambda$.
If $u \in E_\lambda \subset \ker L$ then it lies in the $(2n+2)$-eigenspace of $\Delta^h$ so $u$ defines a normalized Hamiltonian vector field $Y = -iu \xi + \desharp u$ by the correspondence in Theorem \ref{teocorresp}.

Now we compute the adjoint action of $X$ on $\lie h_N$.
\begin{prop}
For $Y$ in the image of $E_\lambda$ it is $[X, Y] = \lambda Y$.
\end{prop}
\begin{proof}
The action of $X$ is given by \eqref{eq:bracket}, namely
\begin{equation}
[X, Y] = -i (Xu-Yh) \xi + [\desharp h, \desharp u].
\end{equation}

Consider the two summands separately. Compute, for $u \in \ker L$,
\begin{align*}
\desharp (\bar L u) 	&= 2i \desharp (J \tilde X \cdot u) \\
				&= 2i [J \tilde X, \desharp u]\\ 
				&= 2i J [ \tilde X, \desharp u]
\end{align*}
where the second equality is due to the fact that $\grad(Kf) = [K, \grad f]$ for any Riemannian manifold, Killing vector field $K$ and function $f$ on it.
So we obtain, if $u \in E_\lambda$,
\begin{align}\label{eq:eigenv}
[\desharp h, \desharp u]		&= [\tilde X, \desharp u] -i[J \tilde X, \desharp u] \nonumber \\ 
						&= -\frac 1 {2i} J \desharp (\bar L u) - \frac 1 2 \desharp (\bar L u)  \nonumber \\
						&= -\desharp( \bar L u)\\ 
						&= \lambda \desharp u. \nonumber
\end{align}

Now note that $Yh = \desharp u \cdot h = \nabla^i u \nabla_i h = \bar X u$. So
\[
-i (Xu - Yh) = -i (X - \bar X) u = -i (L - \bar L) u =- i \lambda u.
\]
Hence $[X, Y] =  \lambda Y$. 
\end{proof}
Consider now the zero eigenspace $E_0 = \ker L \cap \ker \bar L$ of $\bar L|_{\ker L}$. Mimicking Tian and Zhu's argument \cite{tz_uniq} we get that for $u \in \ker L \cap \ker \bar L$ it is $L(\Re u) = L(\Im u) = 0$, so $E_0$ splits as $E_0' \oplus E_0''$, the space of real valued and purely imaginary functions in $\ker L \cap \ker \bar L$.
This corresponds to a splitting of the image of $E_0$ as $\lie h_0 = \lie k_0 \oplus \lie k_0'$. We have

From now on, if $\lie p \subseteq \lie h$ is a Lie subalgebra containing $\xi$, we let $\bar{\lie p}$ denote the quotient $\lie p / \xi$.
The following lemma is basically \cite[Lemma~2.11]{vancov}.
\begin{lemma} \label{lemmaz0}The space $\bar{\lie k_0}$ is formed by the fields whose real part is transversally Killing.
\end{lemma}

Our goal now is to write a decomposition of some algebra of transversally holomorphic vector fields, analogously to the case of extremal Sasakian metrics.
A natural Lie algebra to consider would be $\lie{fol}(M, \xi,  J)$. This is infinite dimensional as it contains the space of sections of the foliation distribution. So its projection onto the normal space of the foliation has been considered in \cite{canonical, vancov}.
On the other hand, in analogy of the K\"ahler-Ricci soliton case, we are interested only in \emph{Hamiltonian} fields which in particular are transversally holomorphic, that is $\lie h \subset \lie{fol}(M, \xi,  J)$. More precisely, we consider the projection $\bar{\lie h}$ onto the normal space, which also is finite dimensional.

As it is said in \cite{FOW}, given a Hamiltonian holomorphic vector field, one can obtain a normalized one by adding a constant multiple of $\xi$, so the space of normalized fields is a set of representatives for the classes of $\bar{\lie h}$.

We have already computed  the action of $X$ on normalized vector fields, so we can get also  the adjoint action of its class  $\bar X \in \bar{\lie h}$.
Recall that the Lie algebra of infinitesimal Sasaki transformations is defined by
\[
\lie{aut}(\mathcal S) = \{ X \in \Gamma(TM) : \lieder_X g = 0, \lieder_X \eta = 0 \}.
\]
Of course $\xi$ is central in it, so it makes sense to consider the quotient $\lie{aut}(\call S)/\xi$. Finally, let $\lie{aut}^T = \{ \bar Y : \lieder_{\bar Y} J = 0, \lieder_{\bar Y} g^T = 0 \}$. We have the following result.

\begin{thm}\label{thm:decomp}On a compact Sasaki-Ricci soliton $\call S$, the finite dimensional Lie algebra $\bar{\lie h}$ admits the decomposition
\[
\bar{\lie h} =   \bar{\lie k_0} \oplus J \bar{\lie k_0} \oplus \bigoplus_{\lambda>0} \bar{\lie h}_\lambda
\]
where  $\bar{\lie k_0}$ is the space in Lemma \ref{lemmaz0} and $\bar{\lie h}_\lambda = \{ \bar Y \in \bar{\lie h}: [\bar X, \bar Y] = \lambda \bar Y \}$.

Moreover $\bar{\lie k_0} \oplus J \bar{\lie k_0}$ is the centralizer in $\bar{\lie h}$ of $\bar X$ and the space $ \bar{\lie k_0}$ can be identified with $\lie{aut}(\mathcal S)/\xi$ and with $\lie{aut}^T$.
\end{thm}
\begin{oss}
In contrast with the similar decomposition in the case of extremal Sasakian metrics, here we do not have any summand corresponding to transversally parallel fields. 
In fact, SRS are Fano and therefore there are no basic harmonic $1$-forms, hence no transversally parallel fields.
\end{oss}

\begin{proof}[Proof of the theorem]
The eigenspace decomposition follows from the  adjoint action of $X$ on $\lie h_N$ computed above. 

For the last statement, on one hand it is clear that a function in $E_0$ induces a field that commutes with $X$ in $\lie h_N$ so its class belongs to the centralizer of $\bar X$.
On the other, a class $\bar Y$ induced by a normalized function $u \in \ker L$ centralizes $X$ if and only if $[\desharp h, \desharp u] = 0$. 
From \eqref{eq:eigenv} we see that in this case we need to have 
\[
\bar L u = \bar{L \bar u}= \Delta^h \bar u - (2n+2) u = \const.
\]
Integrating it with the weighted measure we see that the constant has to be zero, so $u \in E_0$.

Let us now prove the statement about the Lie algebra of infinitesimal Sasaki transformations. 
In \cite[Lemma 2.11]{vancov} it is proved that for a purely imaginary basic function $f$ the field $V= \Re \desharp f$ lifts to a vector field $\tilde V \in \lie{aut}(\call S)$ and conversely a vector field $\tilde V \in \lie{aut}(\call S)$ is such that its projection is the real part of $\desharp f$ for the purely imaginary function $f=i \eta(\tilde V)$.
 The $\desharp$-image of purely imaginary functions, followed by the projection onto the normal bundle is exactly $\bar{\lie k_0}$ of Lemma \ref{lemmaz0}.
 
 There is a well known exact sequence, see e.g. \cite{monoBG},
 \[
 0 \rightarrow \{ \xi \} \rightarrow \lie g' \rightarrow \lie{aut}^T \rightarrow H_B^1(M) \simeq H^1(M, \R).
 \]
 that means that the first (basic) cohomology group is an obstruction to the identification $\lie{aut}(\call S)/\xi \simeq \lie{aut}^T$ and in the transversal Fano case there is no such obstruction.
  \end{proof} 
\begin{oss}
In order to be consistent with analogue results in the literature for the Sasaki extremal case \cite{canonical, vancov} or more generally transversely K\"ahler harmonic foliations as in \cite{NishTond} we have stated Theorem \ref{thm:decomp} for the quotient algebra $\lie h/\xi$.
The computation of the adjoint action together with Lemma \ref{lemmaz0} prove that a similar decomposition holds for the finite dimensional Lie algebra $\lie h_N$ as well although the Lie algebra of infinitesimal Sasaki transformations cannot fit in the picture since it is not contained in $\lie h_N$.
\end{oss}
\section{Deformations of Sasaki-Ricci solitons} \label{sec:deform}
\subsection{Generalized Sasaki-Ricci solitons} 
Here we extend to the Sasakian setting the result obtained for K\"ahler-Ricci solitons by Li in \cite{Li}.

There is a wider class of metric that includes Sasaki-Ricci solitons. In the following let $\theta_X$ be the potential (up to constant) of a Hamiltonian holomorphic vector field and let $\laplB$ denote the $\dB$-Laplacian acting on basic functions.
\begin{defi} \label{def:genSRS}
A \emph{generalized Sasaki-Ricci soliton} (generalized SRS for short)  on compact $M^{2n+1}$ is a Sasakian metric whose transverse scalar curvature satisfies 
\begin{equation}\label{eq:genSRS}
s^T - s_0^T = - \laplB \theta_X
\end{equation}
for a Hamiltonian holomorphic vector field $X$ and where $s_0^T = \frac 1 {\vol(M)} \int_M s^T \mu$ is the average transverse scalar curvature of $g$ and $\mu = (\frac 1 2 d\eta)^n \wedge \eta$ is the volume form as before.
\end{defi}
This is of course a generalization of Sasaki-Ricci solitons.

An ``imaginary'' version of  Lemma \ref{lemmaz0} can be stated as follows. See \cite{LBS} for the K\"ahlerian counterpart.
\begin{lemma} Let $(M, \call S)$ be a Sasakian manifold.
The transverse field $X$ can be expressed as $X = \desharp f$ for a real basic function $f$ if, and only if, $V = \Im \desharp f$ is Killing for $g^T$. In this case $V$ lifts to $\tilde V \in \lie{aut}(\call S)$. Conversely, if $\tilde V \in \lie{aut}(\call S)$ then its projection is $\Im \desharp f$ for the real function $f = -\eta(\tilde V)$.
\end{lemma}
\begin{proof}
Let $X = \desharp f$ with $f$ real and let $V = \Im X$. Then we notice, since $\omega^T$ real, that $V$ has $f/2$ as Hamiltonian function with respect to the transverse symplectic form. Indeed 
\begin{equation*}
\iota_V \omega^T = \frac 1 {2i} (\iota_{\desharp f} - \iota_{\bar{\desharp f}}) \omega^T = \frac 1 2 (\debar f + \de f) = \frac 1 2 df.
\end{equation*}
So $\lieder_V \omega^T = 0$. Conversely, let $X = \desharp f = iV + J V$ with $V$ transversally Killing and $f = u+iv$. Then taking the imaginary part of the equation $\iota_{\desharp f} \omega^T = i (\debar u + i \debar v)$ we have $\iota_V \omega^T = \debar u$ and hence $\de \debar u = 0$ so $u$ is constant.
In this case, to extend $V$ to a $\tilde V \in \lie{aut}(\call S)$ we need to find a function $a$ such that $\tilde V = a \xi + V$ is contact. This means
\[
\lieder_{\tilde V} \eta = da + \iota_V d\eta = da + df = 0
\]
so we see that we can lift $V$ to $\tilde V$ if, $a = -f$.
Conversely $\tilde V =-f \xi + V$ being contact means that $0=d(\eta(\tilde V)) + \iota_V d\eta = -df + \iota_V d\eta$ hence $V= \Im \desharp f$.
\end{proof}
\subsection{Main result}
As in \cite{vancov}, let us now fix a compact connected $G \subseteq \Aut(\mathcal S)^0$ with Lie algebra $\lie g$ with center $\lie z$ and such that $\xi \in \lie z \subseteq \lie g$.
Then it makes sense to consider $\bar{\lie z}$, whose elements are transversally Killing and imaginary parts of projected  Hamiltonian holomorphic fields whose potentials are $G$-invariant.

We want to apply to $\call S$ a deformation as in \eqref{alldefo}, so parameterized by $(\paramt) \in \call B \times \lie z \times \CinfB(M)$.

Start with a basis $\{ v_0 = \xi, v_1, \ldots, v_d \}$ of $\lie z$ and let $X_j = i \bar v_j + J \bar v_j$ in a way that $\Im X_j = \bar v_j$. Consider the functions (depending on the Sasakian structure) 
\begin{equation} \label{basep}
p_\paramt^0 = 1 \text{ and }p_\paramt^j = -\eta_\paramt(v_j).
\end{equation}
Let $Y_j = -i p_g^j \xi + \desharp p_g^j$. It is Hamiltonian holomorphic and the functions $p_g^j$ acts as a holomorphy potential like the K\"ahler case.

Let $\call H_g^{\lie p}$ be, for any Lie algebra $ \xi \in \lie p \subseteq \lie{aut}(\call S)$, the space of functions $u$ such that $\desharp u$ lies in the complexified quotient $ \bar{\lie p}^\C$.
 
A metric defines an orthogonal splitting of $H^k(M)^G$ as
\[
 H^k(M)^G = \call H_g^{\lie z} \oplus W_g
 \]
Let $\Pi^\perp_g$ be the projection onto $W_g$.
We will consider the function
\begin{equation}
S(\paramt) := \Pi_g^\perp \Pi_\paramt^\perp G_\paramt(s_\paramt^T - s_\paramt^0) 
\end{equation}
where $G_\paramt$ is the Green operator of $\dB$ with respect to the metric $g_\paramt$.
For the metric $g_\paramt$ to be a generalized SRS we need $ G_\paramt(s_\paramt^T - s_\paramt^0)$ to lie in $\call H_\paramt^{\lie z} := \call H_{g_\paramt}^{\lie z}$, so $S(\paramt)=0$.
Since $\ker (\Pi_g^\perp \circ  \Pi_\paramt^\perp) = \ker \Pi_\paramt^\perp $ if the deformation is small enough, we have 
\[
S\colon \call V \subseteq  \call B \times \lie z \times H^k(M)^G \rightarrow W_g
\]
for $\call V$ a neighborhood of $(0,0,0)$.
Let us compute the derivatives of $S$. The derivative along $\phi$ behaves as in the K\"ahler case.

\begin{lemma}[\cite{canonical, FOW}]As in the K\"ahler case, the variation of the scalar curvature under type II deformations is
\begin{equation}\label{eq:Dscal}
D_\phi s_\phi^T|_{\phi=0} (\psi) = - \frac 1 2 \laplB^2 \psi - 2(\rho^T, i \deB \debarB \psi).
\end{equation}
Moreover, the average scalar curvature is constant.
\end{lemma}

The derivative of $S$ along $\phi$ is
\begin{equation}\label{DS}
D_\phi S \restr		= \Pi_g^\perp (D_\phi \Pi_\paramt) \restr \theta_X + \Pi_g^\perp D_\phi A_\paramt \restr
\end{equation}
where $A_\paramt = G_\paramt(s_\paramt^T - s_\paramt^0)$.

Let $\{ f_\paramt^j \}$ be obtained from \eqref{basep} via the Gram-Schmidt procedure with respect to the weighted $L^2$ product 
\[
\langle f, h \rangle_\paramt = \int_M fh e^{\theta_X} \mu_\paramt.
\]
In particular it is $f_\paramt^0 = \vol_\paramt^{-1/2}$ and 
\begin{equation} \label{f1}
f_\paramt^1 = \frac{ p_\paramt^1 - \langle p_\paramt^1, 1 \rangle_\paramt \frac{1}{\vol_\paramt}}{\|p_\paramt^1 - \langle p_\paramt^1, 1 \rangle_\paramt \frac{1}{\vol_\paramt}\|_\paramt}.
\end{equation}
Now we have
\begin{align*}
 (D_\phi \Pi_\paramt) \restr \theta_X	&= \sum_{j=0}^d \langle f_{0,0}^j, \theta_X \rangle_g D_\phi f_\paramt^j \restr \pmod {\ker \Pi_g^\perp} \\
 								&= D_\phi f_\paramt^1 \restr \pmod {\ker \Pi_g^\perp}\\
\end{align*}
as $ \langle f_{0,0}^j, \theta_X \rangle_g = \delta_{1,j}$.
Deriving \eqref{f1} we have 
\begin{equation}
D_\phi f_\param^1 \restr (\psi)= D_\phi p_\paramt^1 \restr (\psi)=X \psi \pmod{\ker \Pi_g^\perp}
\end{equation}
 because $\Pi_g^\perp$ kills the constants.

Now, deriving the relation $(-2 i \partial \bar \partial A_\paramt, \frac 1 2 d\eta_\param) = s_\param^T - s_\param^0$ we have
\[
(-2 i \de \debar D_\phi A_\param \restr(\psi), \frac 1 2 d\eta) + (-2 i \de \debar A_\param, \frac 1 2 dd^c \psi) = - \frac 1 2 \Delta^2 \psi - 2(\rho^T, i \de \debar  \psi).
\]
So we have
\begin{align} \label{DA}
\Delta_g D_\phi A_\param \restr(\psi)	&= 2(i \de \debar \theta_X, i \de \debar \psi) - \frac 1 2 \Delta_g^2 \psi - 2(\rho^T, i \de \debar \psi) \nonumber \\
								&= (2n+2) \Delta_g \psi - \frac 1 2 \Delta_g^2 \psi
\end{align}
where we have used that $2(i \de \debar \theta_X, i \de \debar \psi) = 2(\rho^T, i \de \debar \psi) + (2n+2)\Delta_g \psi$ from the SRS equation. So we get
\begin{equation} \label{DA}
D_\phi A_\param \restr(\psi) = G_g \biggl ( -\frac 1 2 \Delta_g^2\psi + (2n+2) \psi \biggr ).
\end{equation}
Using \eqref{DA} in \eqref{DS} becomes
\begin{align}\label{DScomputed}
D_\phi S \restr(\psi) 	&= -\Pi_g^\perp \biggl (\frac 1 2 \Delta_g \psi -  X\psi - (2n+2) \psi \biggr ) \nonumber \\
				&= -\Pi_g^\perp (\bar L \psi)
\end{align}
where $\bar L$ is the operator defined in Section \ref{sec:inftransf}.
The derivative with respect to $\alpha$ can be computed similarly for the first summand but the computation \eqref{DA} cannot be repeated as the foliation changes. We then have
\begin{equation}
(D_\alpha \Pi_\param)\restr(\beta) = \Pi_g^\perp \biggl ( \eta(\beta) \theta_X + D_\alpha A_\param \restr(\beta) \biggr).
\end{equation}
\begin{oss}
A case where we can compute also the second term is given when $\dim \lie z = 1$, that is when there is room only for Tanno $D$-homothetic deformations \cite{tanno}. 
In this case $\alpha = a \xi$, so 
\begin{align*}
\xi_a 	&= (a+1) \xi\\
\eta_a	& = \frac 1 {a+1} \eta\\
g_a^T 	&= \frac 1 {a+1} g^T\\ 
 \Phi_a	& = \Phi.
\end{align*}
We have $\mu_a = (a+1)^{-(n+1)} \mu$ and so $s_a^T - s_a^0 =  (a+1)(s^T - s^0)$.

Being the characteristic foliation unchanged it makes sense to replicate the computation \eqref{DA} and we get $\Delta_a A_a = (a+1) (s^T - s^0)$ so $\Delta_g (D_a A_a) \restr - \Delta_g \theta_X = s^T - s^0 $. Using the generalized SRS equation we have that $(D_a A_a) \restr$ is a constant and hence killed by $\Pi_g^\perp$. So finally
\[
(D_\alpha S)\restr(\beta) = \Pi_g^\perp \biggl ( \eta(\beta) \theta_X \biggr).
\]
\end{oss}
In any case the derivative will assume the form
\begin{equation}
DS\restr(\beta, \psi) = - \Pi_g^\perp \biggl (\bar L \psi - \eta(\beta) \theta_X -  D_\alpha A_\param \restr(\beta) \biggr)
\end{equation}
\begin{prop} \label{prop:sur}
The map $DS\restr(0,0, \cdot):  H^k(M)^G \rightarrow W_g$ is surjective.
\end{prop}

\begin{proof}
Let $\psi'$, orthogonal to the image of $DS\restr$, be a representative of a class in the cokernel. It must be then $\langle \psi', \Pi_g^\perp \bar L \psi \rangle =\langle \psi', \bar L \psi \rangle= 0$ for all $\psi \in H^k(M)^G$ and the scalar product is the weighted one. It follows that $\bar L^* \psi' = \bar L \psi' =0$.

This means that $\desharp_g \bar{\psi'}$ is a transversally holomorphic $G$-invariant vector field, so it belongs to $\lie z^\C$, a contradiction.
\end{proof}

Let now $K = \ker DS\restr \subseteq \call B \times \lie z \times H^k(M)^G$, let $\pi$ be the projection onto it and consider the map
\begin{equation} \label{mapG}
G:= S \times \pi: \mathcal V \rightarrow W_g \times K
\end{equation}
which is such that $DG\restr(0, \cdot)$ is invertible. We can now state the SRS analogue of  \cite[Thm.~4.7]{vancov}.
\begin{thm} \label{thm:defoSRS}
Let $\call S = (\eta, \xi, \Phi, g)$ be a  Sasaki-Ricci soliton, $G \subseteq \Aut(\call S)^0$ be a fixed compact connected subgroup of Sasaki transformations  of $(M, \call S)$ and such that $\xi \in \lie z \subseteq \lie g$. Let $(\call F_\xi, J_t)_{t \in \call B}$ be a $G$-equivariant deformation.

Then there is a neighborhood $ \call V$ of $(0,0,0) \in  \call B \times \lie z \times \CinfB(M)^G$ such that 
\[
\call E = \biggl \{ (\paramt) \in \call V: g_\paramt \text{ is a generalized SRS} \biggr \} 
\]
is a smooth manifold of dimension $\dim \call B +\dim \lie z$.
\end{thm}
\begin{proof}
We start with a SRS so $M$ is positive and the deformation is of $(1,1)$-type.

The map $G$ of \eqref{mapG} is under the assumptions of \cite[Thm.~17.6]{elliptic} so we have a neighborhood $\call N$ of zero in $\call B \times \lie z$ and a function $f: \mathcal N \rightarrow C_B^\infty(M)^G$ such that $S(t, \alpha,f(t, \alpha)) = 0$ for all $(t, \alpha) \in \call N$.
So the space of solutions of $S=0$ is parameterized by $(t, \alpha)$ and hence it has dimension $\dim \call B + \dim \lie z$.
\end{proof}
\bibliography{biblio}{}
\bibliographystyle{amsplain}
\end{document}